\documentclass[11pt,twoside,reqno]{amsart}

\usepackage{amsmath, amssymb, amsthm, enumitem, esint}
\usepackage[hyperindex]{hyperref}

\usepackage{hyperref}
\hypersetup{bookmarksdepth=3}

\newcounter{mtheorem}
\newtheorem{mtheorem}[mtheorem]{Theorem}

\newcommand{\noopsort}[1]{}

\newcommand{\IR}{\mathbb{R}}
\newcommand{\lb}{\linebreak[1]}
\newcommand{\IC}{\mathbb{C}}

\newcommand{\IZ}{\mathbb{Z}}

\newcommand{\IF}{\mathbb{F}}
\newcommand{\NN}{\mathcal{N}}

\newcommand{\RR}{\mathcal{R}}

\newcommand{\XX}{\mathcal{X}}

\newcommand{\eps}{\varepsilon}

\DeclareMathOperator{\loc}{loc}

\DeclareMathOperator{\Ric}{Ric}
\DeclareMathOperator{\Rm}{Rm}

\DeclareMathOperator{\inj}{inj}

\DeclareMathOperator{\vol}{vol}

\newcommand{\EMPTY}[1]{}

\newtheorem{Lemma}[equation]{Lemma}

\newtheorem{Claim}[equation]{Claim}

\theoremstyle{definition}

\theoremstyle{remark}

\numberwithin{equation}{section}

\begin{document}

\title{A new complete two-dimensional shrinking gradient K\"ahler-Ricci soliton}
\date{\today}

\author{Richard H. Bamler}
\address{Department of Mathematics, University of California Berkeley, CA 94720, USA}
\email{rbamler@berkeley.edu}

\author{Charles Cifarelli}
\address{D\'{e}partement de Math\'{e}matiques, Universit\'{e} de Nantes, 2 rue de la Houssini\`{e}re,  BP 92208, 44322 Nantes cedex 03, France}
\email{charles.cifarelli@univ-nantes.fr}

\author{Ronan J.~Conlon}
\address{Department of Mathematical Sciences, The University of Texas at Dallas, Richardson, TX 75080}
\email{ronan.conlon@utdallas.edu}

\author{Alix Deruelle}
\address{Sorbonne Universit\'e and Universit\'e de Paris, CNRS, IMJ-PRG, F-75005 Paris, France}
\email{alix.deruelle@imj-prg.fr}

\date{\today}

\begin{abstract}
We prove the existence of a unique complete shrinking gradient K\"ahler-Ricci soliton with bounded scalar curvature on the blowup of $\mathbb{C}\times\mathbb{P}^{1}$ at one point. This
completes the classification of such solitons in two complex dimensions.
\end{abstract}

\maketitle

\section{Introduction}

\subsection{Overview}
Shrinking Ricci solitons are natural generalizations of Einstein manifolds with positive scalar curvature.
As such, their study and classification has become a central topic in both Riemannian and K\"ahler geometry.
Indeed, shrinking K\"ahler-Ricci solitons are known to exist on certain Fano manifolds that have obstructions to the existence of a K\"ahler-Einstein metric \cite{zhuu}.
Moreover in Ricci flow, shrinking Ricci solitons comprise candidate singularity models, thereby making their study crucial, for example,
in potentially implementing higher-dimensional surgery constructions.

While in (real) dimensions 2 and 3, a full classification of shrinking Ricci solitons has been achieved (these are Euclidean or quotients of spheres $S^2, S^3$ or the cylinder $S^2 \times \IR$) \cite{Hamilton_3_manifolds,Perelman2},
the situation is far from clear in (real) dimension 4.
Apart from the obvious examples ($\mathbb{R}^{4}$ or quotients of $S^4, S^3 \times \IR$, and $S^2 \times \IR^2$) and
the ten del Pezzo surfaces \cite{Tiant, zhuu}, the only other example in this dimension was found by Feldman, Ilmanen, and Knopf \cite{FIK} nearly $20$ years ago.
Its construction was, in part, possible due to its cohomogeneity one $U(2)$-symmetry,
which allowed the reduction of the soliton equation to a system of ODEs (the solution of which nevertheless still posed a non-trivial problem).

In this paper we prove the existence of a new complete non-compact shrinking K\"ahler-Ricci soliton in complex dimension~2.
This soliton is invariant under a real two-dimensional torus action, hence it has cohomogeneity~2.
Its underlying complex manifold is biholomorphic to the blowup of $\IC \times \mathbb{P}^1$ at one point,
a model that was already identified as the last remaining candidate in previous work \cite{ccd} of the latter three named authors.
We therefore complete the classification of shrinking K\"ahler-Ricci solitons in complex dimension~2.
In addition, we employ a novel approach, in which we construct the soliton indirectly as a blowup limit of a specific K\"ahler-Ricci flow on a compact manifold.
We use recent estimates obtained by the first author \cite{Bamler_HK_entropy_estimates,Bamler_RF_compactness,Bamler_HK_RF_partial_regularity} combined with K\"ahler-Ricci flow techniques to control the singularity formation of this flow.

\subsection{Main results}\label{main-results}
We begin by recalling the main definitions.

A \emph{Ricci soliton} is a triple $(M,\,g,\,X)$, where $M$ is a Riemannian manifold endowed with a complete Riemannian metric $g$
and a complete vector field $X$, such that
\begin{equation}\label{hot}
\Ric_{g}+\frac{1}{2}\mathcal{L}_{X}g=\lambda g
\end{equation}
for some $\lambda\in\mathbb{R}$. The vector field $X$ is called the
\emph{soliton vector field}. If $X=\nabla^{g} f$ for some smooth real-valued function $f$ on $M$,
then we say that $(M,\,g,\,X)$ (or $(M,\,g,\,f)$) is \emph{gradient}. In this case, the soliton equation \eqref{hot}
becomes
$$\Ric_{g}+\nabla^{2}f - \lambda g=0,$$
and we call $f$ the \emph{soliton potential}. In the case of gradient Ricci solitons, the completeness of $X$ is guaranteed by the completeness of $g$
\cite{Zhang-Com-Ricci-Sol}.

Let $(M,\,g,\,X)$ be a Ricci soliton. If $g$ is K\"ahler and $X$ is real holomorphic, then we say that $(M,\,g,\,X)$ is a \emph{K\"ahler-Ricci soliton}. Let $\omega$ denote the K\"ahler
form of $g$. If $(M,\,g,\,X)$ is in addition gradient, then \eqref{hot} may be rewritten as
\begin{equation*}
\rho_{\omega}+i\partial\bar{\partial}f=\lambda\omega,
\end{equation*}
where $\rho_{\omega}$ is the Ricci form of $\omega$ and $f$ is the soliton potential.

Finally, a Ricci soliton and a K\"ahler-Ricci soliton are called \emph{steady} if $\lambda=0$, \emph{expanding}
if $\lambda<0$, and \emph{shrinking} if $\lambda>0$ in \eqref{hot}.
One can always normalise $\lambda$, when non-zero, to satisfy $|\lambda|=1$. We henceforth assume that this is the case.

Our main result is now the following.

\begin{mtheorem}[Existence and uniqueness]\label{mainthm}
Up to automorphism, there exists a unique complete shrinking gradient K\"ahler-Ricci soliton with bounded scalar curvature on $\operatorname{Bl}_{x}(\mathbb{C} \times \mathbb{P}^1)$, that is, the blowup of $\mathbb{C}\times\mathbb{P}^{1}$ at a fixed point $x$ of the standard real torus action on $\mathbb{C}\times\mathbb{P}^{1}$. Moreover, this soliton is invariant under
the induced real torus action and appears as a parabolic blowup limit of the K\"ahler-Ricci flow.
\end{mtheorem}

The soliton of Theorem~\ref{mainthm} is constructed as a singularity model of a specific Ricci flow.
More precisely, we consider the blowup $N := \operatorname{Bl}_{x}(\mathbb{P}^{1} \times \mathbb{P}^1)$ of $\mathbb{P}^{1} \times \mathbb{P}^1$ at
one point and show that there is a toric K\"ahler-Ricci flow that contracts the exceptional divisor and exactly one other $(-1)$-curve at the singular time $T>0$.
The volume of $N$ close to the singular time is $\sim (T-t)$.
Using the estimates from \cite{Bamler_HK_RF_partial_regularity}, we analyze possible blowup models of this flow.
Thanks to the toricity of the flow and the topology of the underlying manifold, we are able to exclude orbifold singularities from appearing in the limit
and show that the singularity is close to a \emph{smooth} shrinking K\"ahler-Ricci soliton at most scales.
By \cite{ccd}, regions that are close enough to such K\"ahler-Ricci solitons must contain a complex curve of self-intersection $0$ or $-1$.
The areas of these curves, which are determined by the K\"ahler class of the evolving metric,
shrink by at most a linear rate. This observation allows us to bound from below the scales at which the flow exhibits closeness to a shrinking K\"ahler-Ricci soliton.
Consequently, we see that the flow must be Type I, i.e., we have curvature bounds of the form $|{\Rm}| \leq C/(T-t)$.
As a result, the singularity formation near every point can be described by a shrinking K\"ahler-Ricci soliton.
Among these, we are able to exclude the soliton of Feldman-Ilmanen-Knopf because the volume of $N$ under the flow converges to zero \cite{ccd}.
This only leaves $\IC \times \mathbb{P}^1$ and the candidate from Theorem~\ref{mainthm} as possible blowup limits.
Complex geometric reasons allow us to further rule out a $\IC \times \mathbb{P}^1$ forming near the exceptional divisor.
This demonstrates that an additional soliton, namely that characterized in \cite{ccd}, must exist.

To each complete shrinking K\"ahler-Ricci soliton $(M,\,g,\,X)$ satisfying \eqref{hot} with $\lambda=1$,
one obtains a solution of the K\"ahler-Ricci flow $\partial_t g(t)=-\operatorname{Ric}_{g(t)}$. Indeed, set $$g(t):=-t\varphi_{t}^{*}g,\quad t<0,$$
where $\varphi_{t}$ is a family of diffeomorphisms generated by the vector field $-\frac{1}{t}X$ with $\varphi_{-1}=\operatorname{id}$, i.e.,
\begin{equation*}
\frac{\partial\varphi_{t}}{\partial t}(x)=-\frac{X(\varphi_{t}(x))}{t},\quad\varphi_{-1}=\operatorname{id}.
\end{equation*}
Then $\partial_t g(t)=-\operatorname{Ric}_{g(t)}$ for $t<0$, and $g(-1)=g$. In this way,
non-flat complete shrinking gradient K\"ahler-Ricci solitons with bounded curvature appear as parabolic rescalings
of finite time Type I singularities of the K\"ahler-Ricci flow \cite{Sesum_conv_to_soliton,Cao_Zhang_conj_h_e,topping, naber}.
In other words, they are models for such singularities and hence this motivates their classification. Theorem \ref{mainthm} completes such a
classification in two complex dimensions.

Indeed, a complete two-dimensional shrinking gradient K\"ahler-Ricci soliton with bounded curvature is either compact, in which case the underlying manifold
 is Fano and the resulting soliton is, up to automorphism \cite{tianzhu1}, K\"ahler-Einstein or the shrinking gradient K\"ahler-Ricci soliton given by \cite{soliton} depending on the Fano manifold in question, or is non-compact
with bounded scalar curvature. Shrinking gradient K\"ahler-Ricci solitons are connected at infinity \cite{munteanu} and in this latter case, there is a dichotomy in the
sense that the scalar curvature of the soliton either tends to zero along every integral curve of $X$, or $X$ has an integral curve along which the scalar curvature
does not tend to zero. In the former case, it follows that the scalar curvature tends to zero globally (cf.~\cite[Lemma 2.7]{ccd}) and hence the soliton is (up to automorphism) either that of Feldman-Ilmanen-Knopf \cite{FIK} on the blowup of $\mathbb{C}^{2}$ at one point or the flat Gaussian shrinking soliton on $\mathbb{C}^{2}$ \cite{cds}. In the latter case, the shrinking soliton is either isometric to the cylinder $\mathbb{C}\times\mathbb{P}^{1}$ or to the shrinking soliton of Theorem \ref{mainthm}.

\begin{mtheorem}[Classification of two-dimensional K\"ahler shrinkers]\label{mainthm2}
Let $(M,\,g,\,X)$ be a complete two-dimensional shrinking gradient K\"ahler-Ricci soliton with bounded scalar curvature. Then either:
\begin{enumerate}[label=\textnormal{(\roman{*})}, ref=(\roman{*})]
  \item $M$ is Fano and $g$ is, up to automorphism, either K\"ahler-Einstein or the shrinking gradient K\"ahler-Ricci soliton on $M$ given by \cite{soliton}, or:
   \item $(M,\,g,\,X)$ is, up to pullback by an element of $GL(2,\,\mathbb{C})$, the flat Gaussian shrinking soliton on $\mathbb{C}^{2}$, or:
   \item $(M,\,g,\,X)$ is, up to pullback by an element of $GL(2,\,\mathbb{C})$, the unique $U(2)$-invariant shrinking gradient K\"ahler-Ricci soliton
of Feldman-Ilmanen-Knopf \cite{FIK} on the total space of $\mathcal{O}(-1)$ over $\mathbb{P}^{1}$, or:
\item $(M,\,g,\,X)$ is, up to automorphism, the cylinder $\mathbb{C}\times\mathbb{P}^{1}$, or:
\item $(M,\,g,\,X)$ is, up to automorphism, the shrinking gradient K\"ahler-Ricci soliton of Theorem \ref{mainthm}.
\end{enumerate}
\end{mtheorem}

Thus, we are left with the following picture. Let $(M,\,g(t))_{t\in[0,\,T)}$ be a K\"ahler-Ricci flow developing a finite Type I singularity when $t=T>0$.
Take a blowup limit $g_{\infty}(t)$ of the rescaled flows $g_j (t) := \lambda^{-2}_j g(T+\nolinebreak\lambda^{2}_jt)$, $t\in[-\lambda^{-2}_{j}T,\,0),$ for some sequence $\lambda_j \to 0$, centered at a point $x \in M$ where the curvature blows up. If $\lim_{t\to T^{-}}\operatorname{vol}(M,\,g(t))>0$, then \cite[Theorem B]{ccd} asserts that this blowup limit is the Feldman-Ilmanen-Knopf shrinking soliton on the blowup of $\mathbb{C}^{2}$ at one point. On the other hand, if there is finite time collapsing at $t=T>0$, i.e., if $\lim_{t\to T^{-}}\operatorname{vol}(M,\,g(t))=0$, then either $\lim_{t\to T^{-}}\operatorname{diam}(M,\,g(t))=0$, which is a ``finite time extinction'', or $\limsup_{t\to T^{-}}\operatorname{diam}(M,\,g(t))>0$. In the former case, \cite{tosatti10} (see also \cite{song23}) asserts that $M$ is Fano and the K\"ahler class of $g(0)$ lies in a positive multiple of $c_{1}(M)$. The work of
Perelman (see \cite{sesum1}) gives us the upper bound $\operatorname{diam}(M,\,g(t))\leq C(T-t)^{\frac{1}{2}}$, which,
for the re-scaled limit $g_{\infty}(t)$, $t<0$, translates as $\operatorname{diam}(M,\,g_{\infty}(t))\leq C(-t)^{\frac{1}{2}},\,t<0$. This latter bound implies that
the rescaled limit is compact, hence being a shrinking soliton, is Fano with its (up to automorphism \cite{tianzhu1}) unique shrinking soliton structure. In the latter case,
the blowup limit cannot be Fano as the compactness of such a manifold would imply that
$\lim_{t\to T^{-}}\operatorname{diam}(M,\,g(t))=0$, a contradiction. By \cite[Theorem B]{ccd}, the blowup limit cannot be
the shrinking soliton of Feldman-Ilmanen-Knopf \cite{FIK}. Hence the only possibility is that the blowup limit is the cylinder
$\mathbb{C}\times\mathbb{P}^{1}$ or the soliton of Theorem \ref{mainthm}, where the latter model corresponds to the contraction of a $(-1)$-curve.

\subsection{Acknowledgements}
The authors wish to thank Song Sun for useful discussions. The first author is supported by NSF grant DMS-1906500.
The second author is supported by the grant Connect Talent ``COCOSYM’’ of the r\'{e}gion des Pays de la Loire and the Centre Henri Lebesgue,  programme ANR-11-LABX-0020-0, the third author is supported by NSF grant DMS-1906466, and the fourth author is supported by grants ANR-17-CE40-0034 of the French National Research Agency ANR (Project CCEM) and ANR-AAPG2020 (Project PARAPLUI).

\section{Proof of Theorem \ref{mainthm}}
Let $\sigma:N=\operatorname{Bl}_{x}(\mathbb{P}^{1} \times \mathbb{P}^1)\to\mathbb{P}^{1} \times \mathbb{P}^{1}$ denote the blowup of $\mathbb{P}^{1} \times \mathbb{P}^{1}$ at
one of the four fixed points $x$ of the standard real torus action on $\mathbb{P}^{1} \times \mathbb{P}^{1}$,
let $E \subset N$ denote the corresponding exceptional divisor, and write $J$ for the complex
structure of $N$. Denote the two fixed points on each factor by $p,\,q\in \mathbb{P}^1$, so that
without loss of generality we can take $x=(p,\,p)\in \mathbb{P}^{1} \times \mathbb{P}^{1}$. Consider $N$ endowed with the induced action of a real torus $T^{2}$
and set $D_1 = \sigma^{-1}(\mathbb{P}^{1} \times \{ q \})$ and $D_2 = \sigma^{-1}(\{ q \} \times \mathbb{P}^{1} )$ so that $D_i$ are torus-invariant divisors in $N$ satisfying
$[D_{i}]^2 = 0$, $[D_{1}].[D_2] = 1$, and $[D_{i}]. [E] = 0$. Then we can choose a K\"ahler metric $g_0$ on $N$ invariant under the torus action such that for the associated K\"ahler-Ricci flow $(g(t))_{t \in [0,T)}$, $T>0$, we have
\[ \operatorname{vol}(D_{1}, g(t)),\,\,\operatorname{vol}(E, g(t)) \xrightarrow[t \to T]{} 0, \quad \operatorname{vol}(D_{2}, g(t)) \not\longrightarrow 0. \]
Indeed, $g(0)$ can be chosen as follows. $H^{1,\,1}(X)=\mathbb{R}^{3}$ is generated by the Poincar\'e duals to $D_{1}$,\,$D_{2}$,\, and $E$.
Using the Nakai-Moishezon criterion for positivity, a computation shows that the K\"ahler cone $\mathcal{K}_N$ of $N$ is given by
\begin{equation*}
	\mathcal{K}_{N} = \left\{ a_1[D_{1}] + a_2[D_{2}] - b[E] \: | \: a_1 > b> 0,\,\, a_2> b> 0 \right\}.
\end{equation*}
Since for any inital K\"ahler $g(0)$, the K\"ahler class of the evolving K\"ahler form $\omega(t)$ satisfies
$$\partial_{t}[\omega(t)]=-[\rho_{\omega(t)}]=-2\pi c_{1}(N),$$
we find that
\begin{equation}\label{hermosa}
  [\omega(t)]=[\omega(0)]-2t\pi c_{1}(N).
\end{equation}
The first Chern class $c_{1}(N)$ of $N$ is given by $2\pi c_1(N) = 2[D_{1}] + 2[D_{2}] - [E]$.
Hence for any initial K\"ahler
form $\omega(0)$ in the K\"ahler class $a_1 [D_{1}] + a_2 [D_{2}] - b[E]\in\mathcal{K}_{N}$, we find that
\begin{equation}\label{kahlerclass}
	[\omega(t)] = (a_1 - 2 t) [D_{1}] + (a_2 - 2 t) [D_{2}]- (b - t) [E].
\end{equation}
In particular, the volume of $D_i,\,i=1,\,2,$ is given by
\begin{equation*}
\begin{split}
	\operatorname{vol}(D_{i}, \omega(t)) &= \int_{D_{i}} [\omega(t)]  = (a_{3 -i} - 2t)[D_{1}].[D_{2}]=a_{3-i} - 2t,
\end{split}
\end{equation*}
whereas the volume of $E$ is given by
\begin{equation*}
\begin{split}
	\operatorname{vol}(E, \omega(t)) &= \int_{E} [\omega(t)]  = -(b-t)[E]^{2}= b-t.
\end{split}
\end{equation*}
Therefore it suffices to choose $g(0)$ to lie in a K\"ahler class with $a_{2}=2b$ and $a_{1}>a_{2}$.
The singular time $T>0$ of the flow is then given by $T=b=\frac{a_{2}}{2}$.
By averaging, we can guarantee that $g(0)$ is invariant under the torus action.

Our Ricci flow has the following property.

\begin{Lemma}\label{volume}
$\vol (M,\,g(t)) \longrightarrow 0$ and $(T-t)^{-2} \vol (M,\,g(t)) \not\longrightarrow 0$.
\end{Lemma}

\begin{proof}
	The volume $\vol (M,\,g(t))$ is given by
		\begin{equation*}
			\vol (M,\,g(t)) = \frac{1}{2} \int_M [\omega(t)]^2 = (a_1 - 2t)(a_2 - 2t) - \frac{1}{2} (b-t)^2 = 4 \left( \frac{a_1}2-t \right)(T-t) - \frac12 (T-t)^2.
		\end{equation*}
		The lemma now follows since $\frac{a_1}2 > \frac{a_2}2 = T$.
\end{proof}

\begin{Lemma} \label{Lem_no_cone}
Let $\Gamma \subset O(4)$ be a non-trivial, finite group acting freely on $S^3$.
Then the following cannot occur:
There are sequences $t_i \in [0,T)$, $\lambda_i > 0$, an exhaustion $U_1 \subset U_2 \subset \ldots \subset (\IR^4 - \{ 0\})/\Gamma$, and diffeomorphisms $\psi_i : U_i \to \psi_i(U_i) = V_i \subset M$ such that $\lambda_i^2 \psi_i^* g(t_i)$ converges locally smoothly to the Euclidean metric on $\IR^4/\Gamma$.
\end{Lemma}

Note that we don't impose any $T^2$-equivariance assumptions on $\psi_i$ in Lemma~\ref{Lem_no_cone}.

\begin{proof}
Suppose that there were such sequences.
Set $g_i := \lambda_i^2 g(t_i)$ and fix a cutoff function $\eta : [0,\infty) \to [0,1]$ with $\eta \equiv 1$ near $0$ and $\eta \equiv 0$ on $[1,\infty)$.
Let $\inj_{g_i} : N \to \IR_+$ denote the injectivity radius with respect to $g_i$ and
define $(u_{\eps} \in C^0(N))_{\eps > 0}$ by
\[ u_{\eps}(p) := \int_N \inj_{g_i} (q) \eta(d(p,q)/\eps) dg_i(q) \bigg/ \int_N  \eta(d(p,q)/\eps) dg_i(q). \]
Next note that $\inj_{g_i} \circ \psi_i$ converges locally uniformly to the injectivity radius on $(\IR^4 - \{ 0\})/\Gamma$, which itself is linear in the radial coordinate.
So if $W := A(0,\tfrac12,2)/\Gamma$ denotes an annulus around the origin in $\IR^4/\Gamma$, then for small enough $\eps > 0$ we have smooth convergence of $(u_\eps \circ \psi_i)|_W$ to a radial function on $W$ without critical points.
So for large $i$, the function $u_\eps$ has a regular level set that contains a component $\Sigma_i \subset M$ diffeomorphic to $S^3/\Gamma$.
By construction, $\Sigma_i$ is invariant under the $T^2$-action on $N$. Since we may choose $\Sigma_i$ to be an arbitrary level set of $u_\eps$, we may assume in addition that the $T^2$-action on $\Sigma_i$ has no fixed point.

Consider now the moment map $\mu : N \to\mathbb{R}^{2}$ of the real $T^{2}$-action, which exists because $N$ is simply connected. As is well-known,
$\mu(N)$ can be realised as a convex polytope $P\subseteq\mathbb{R}^{2}$ with vertices $(-1,0),\,(-1,1),\,(1,1), (1,-1),\,(0,-1)$, with the facet joining $(-1,\,0)$ and $(0,\,-1)$
corresponding to $E$. $P$ is the orbit space of the $T^{2}$-action. Indeed, the pre-image under $\mu$ of points in the interior of $P$ are generic $T^{2}$-orbits, the pre-image under $\mu$ of points along the facets are $S^{1}$-orbits, and the pre-images of the vertices are points, corresponding to fixed points of the $T^{2}$-action.
The image $\mu (\Sigma_i)\subseteq P$ defines a smooth path whose endpoints lie in the facets of $P$; it is not closed, because otherwise $\pi_1(\Sigma_i)$ would be infinite.

Let $F_1, F_2$ denote the two (not necessarily distinct) facets of $P$ which contain the endpoints of $\mu(\Sigma_i)$.
Since $\Sigma_i$ does not contain any fixed point of the torus action, the two endpoints of $\mu(\Sigma_i)$ must lie in the interior of $F_{1}$ and $F_{2}$.
These facets are defined by the linear equalities $\langle \nu_j, x \rangle = -1$, where $\nu_j \in \{ \pm (1,\,0), \,  \pm (0,\,1),  \, (1,\,1) \}$.
First suppose that $F_1$ and $F_2$ are parallel. Then the moment map $\mu$ contracts at each endpoint the same $S^1$-factor of the $T^{2}$-fibration in the interior
so that $\Sigma_i$ defines a trivial $S^1$-bundle over $S^2$, i.e., $\Sigma_i$ must be diffeomorphic to $S^2 \times S^1$. This contradicts the fact that $\Sigma_i \cong S^3/\Gamma$.
Next, suppose that the normals $\nu_1, \nu_2$ are linearly independent. Then running through the various possibilities for $\nu_j$,
we see that the cone $C=\{ x \in \IR^2 \: | \: \langle \nu_j, x \rangle \geq 0, \, j = 1,2 \}\subseteq\mathbb{R}^{2}$
satisfies the Delzant condition \cite[Definition 7]{charlie} so that the complex cone defined by $C$
is biholomorphic to $\mathbb{C}^{2}$. As such, $\Sigma_{i}$ must be diffeomorphic to $S^{3}$, thereby contradicting our assumption that $\Gamma$ is non-trivial.
Notice that $N$ has precisely three invariant $(-1)$-curves, namely those corresponding to the facets of $P$ with inner normals
 $(1,0), \, (1,1)$, and $(0,1)$, respectively. Therefore if $F_1, F_2$ are not consecutive, then they must be separated by some facet of $P$ corresponding to one of these $(-1)$-cuves. Blowing down this
 particular curve, we see that $\Sigma_i$ descends to a hypersurface $\tilde{\Sigma}_i$ inside a toric surface $ \widetilde{N}$ biholomorphic to either
 $\mathbb{P}^1 \times \mathbb{P}^1$ or to the blowup of $\mathbb{P}^2$ at one point.
  Choosing an invariant K\"ahler metric on $\widetilde{N}$ in an appropriate cohomology class, we obtain a new moment map $\tilde{\mu}:\widetilde{N} \to \IR^2$
  with image a convex polytope $\tilde{P}$ such that the endpoints of $\tilde{\mu}(\tilde{\Sigma}_i)$ lie in consecutive edges $\widetilde{F}_1, \widetilde{F}_2$ of $\tilde{P}$
  with the same inner normals as $F_1, F_2,$ respectively. The previous argument then gives that $\tilde{\Sigma}_i$, and hence $\Sigma_i$, is again diffeomorphic to $S^3$,
  a contradiction.

Alternatively, one may conclude the argument in the following way. If $\Sigma_{i}\cap E = \emptyset$, then $\Sigma_{i}$ descends to a hypersurface $\Sigma_{i}'
\subset\mathbb{P}^{1}\times\mathbb{P}^{1}$ invariant under the torus action, which must either be an invariant $S^3$ or an $S^{2}\times S^{1}$
in $\mathbb{P}^{1}\times\mathbb{P}^{1}$. If $\Sigma_{i}\cap E \neq \emptyset$, then $\Sigma_{i}$
intersects one of $D_1, D_2$, the strict transform of $\mathbb{P}^{1}\times\{p\}$ or $\{p\} \times\mathbb{P}^{1}$, or $E$ itself in a separate point.
As we have seen, if $\Sigma_i$ intersects $E$ in two distinct points, then $\Sigma_i \cong S^2 \times S^1$, contradicting the fact that $\Sigma_i \cong S^3/\Gamma$.
Suppose that $\Sigma_i$ meets either $D_1$ or $D_2$. By symmetry, we may assume without loss of generality that $\Sigma_i \cap D_2 \neq \emptyset$. Then $\Sigma_{i}$
is diffeomorphic to the boundary of a small tubular neighborhood of the strict transform of
$\mathbb{P}^{1}\times\{p\}$ in $N$, which has self-intersection $-1$. Thus $\Sigma_{i}$ must be diffeomorphic to $S^{3}$, contradicting
our assumption that $\Gamma$ is non-trivial. To treat the last possibility, assume without loss of generality again that $\Sigma_i$ meets the strict transform of $\{p \} \times \mathbb{P}^{1}$.
Then $\Sigma_i$ is diffeomorphic to the boundary of a small tubular neighborhood of a fixed point of the $T^2$-action, and so we see once again that $\Sigma_{i}$
 must be diffeomorphic to $S^{3}$, contradicting
again our assumption that $\Gamma$ is non-trivial.
\end{proof}

\begin{Lemma} \label{Lem_cylinder_limit}
Let $(M,\,g,\,X)$ be a two-dimensional complete shrinking gradient K\"ahler-Ricci soliton with $X=\nabla^{g}f$ for a smooth
real-valued function $f:M\to\mathbb{R}$. Suppose that there are sequences $p_i \in M$ and $\lambda_i \to 0$ such that we have smooth Cheeger-Gromov convergence
\[ (M,\lambda_i^{-2} g,p_i) \xrightarrow[i\to\infty]{CG} (M_\infty, g_\infty,p_\infty) \]
with $(M_\infty, g_\infty)$ a shrinking gradient Ricci soliton. Then $(M_\infty, g_\infty)$ is, up to automorphism, either the cylinder $\mathbb{C} \times \mathbb{P}^1$ or Euclidean space $\IC^2$.
\end{Lemma}

\begin{proof}
Let $f_\infty$ be a soliton potential for $(M_\infty,\,g_{\infty})$.
Rescaling the soliton equation of $(M,g,f)$ yields for all $k \geq 0$,
\begin{equation} \label{eq_Hess_f_bound}
 |\nabla^{k+2,\lambda_i^{-2} g} f |_{\lambda_i^{-2} g} \leq |\nabla^{k,\lambda_i^{-2} g} {\Ric}_{\lambda_i^{-2} g} |_{\lambda_i^{-2} g} + C_k \lambda_i^2.
\end{equation}
So if $|\nabla^{\lambda_i^{-2} g} f| (p_i)$ remains bounded, then after passing to a subsequence, we have local smooth convergence
\[ f - f(p_i) \xrightarrow[i\to\infty]{C^\infty_{\loc}} f'_\infty \]
for some $f'_\infty \in C^\infty(M_\infty)$ with
\[ \Ric_{g_\infty} + \nabla^2 f'_\infty = 0. \]
This implies that
\[ \nabla^2 (f'_\infty - f_\infty) =  g_\infty. \]
Therefore $f'_\infty - f_\infty$ is proper and attains a unique minimum at some point $x' \in M_\infty$ and we have that $f'_\infty - f_\infty = \tfrac12 d^2(x',\,\cdot)$.
Moreover, the flow of the gradient vector field $\nabla(f'_\infty - f_\infty)$ consists of homotheties.
This implies that $(M_\infty, g_\infty)$ is isometric to $\IC^2$.

Next, suppose that $a_i := |\nabla^{\lambda_i^{-2} g} f| (p_i) \to \infty$.
Then by \eqref{eq_Hess_f_bound}, we find that $a_i^{-1} (f-f(p_i))$ converges locally smoothly to a non-trivial function on $M_\infty$ with vanishing Hessian.
This implies that the limit splits off a line.
Due to the K\"ahler condition, the limit must split off another line, hence must be isometric to $\IC \times \mathbb{P}^1$ or $\IC^2$.
\end{proof}

\begin{Lemma}\label{bonita}
Let $(M', J', g')$ be a (not necessarily complete) K\"ahler surface.
Suppose that $C'_0 \subset M'$ is a $J'$-holomorphic curve in $M'$ biholomorphic to $\mathbb{P}^1$ with self-intersection $-1$ or $0$.
Then there exists $\eps (M', J', g', C'_0) > 0$ such that the following holds true.
Suppose that there is a diffeomorphism onto its image $\psi : M' \to N$ such that for some $\lambda > 0$ and $t \in [0,T)$ we have
\begin{equation} \label{eq_ggp_JJp}
 \Vert \lambda^{-2} \psi^* g(t) - g'  \Vert_{C^{[1/\eps]}} \leq \eps, \qquad \Vert  \psi^* J - J'  \Vert_{C^{[1/\eps]}} \leq \eps,
\end{equation}
where the second bound can be dropped if $(M',g')$ is isometric to an open subset of $\IC \times \mathbb{P}^1$.
Then
\begin{equation} \label{eq_lambda_bound}
\lambda\geq\left(\frac{\pi}{V}\right)^{\frac{1}{2}}(T-t)^{\frac12},
\end{equation}
 where $V=\operatorname{vol}(C_{0}',\,g')$.
\end{Lemma}

Note that we will mainly be interested in Lemma \ref{bonita} in the case where $(M',\,g')$ is an open subset of a complex curve in $\mathbb{C}\times\mathbb{P}^{1}$ or of the shrinking gradient K\"ahler-Ricci soliton of Feldman-Ilmanen-Knopf \cite{FIK}. In these cases, the assumption of the lemma is inferred by pointed Cheeger-Gromov closeness of these model spaces modulo rescaling by $\lambda^{-1}$.

\begin{proof}
If $(M', g')$ is isometric to an open subset of $\IC \times \mathbb{P}^1$, then $J_\eps := \psi^*J$ converges locally smoothly to a $g'$-parallel complex structure $J'_0$ on $M'$ as $\eps \to 0$.
By \cite[Claim 3.3]{ccd}, $J'_0$ coincides with $J'$ up to a sign on each factor, so in this case we may assume without loss of generality that $J'_0 = J'$.
Thus, after possibly adjusting $\eps$, we may assume that both bounds in \eqref{eq_ggp_JJp} hold.

Let $\alpha\in\{-1,\,0\}$ denote the self-intersection of the $J'$-holomorphic curve $C'_0$ in $M'$.
Then by \cite[Corollary 2.3]{ccd} when $\alpha=0$, and
\cite[Corollary 2.4]{ccd} when $\alpha=-1$, and regularity theory \cite[Proposition 3.3.5 and Section B.4]{dusa2}, there exists a $J_{\eps}$-holomorphic curve $C'_{\eps}$ in $M'$ with
self-intersection $\alpha$ converging smoothly to $C'_{0}$ as $\eps\to 0$. Since $\lambda^{-2}\psi^{*}g(t)\to g'$ and $C'_{\eps}\to C'_{0}$ in $C^{1}$ as $\eps\to0$, we can assert that
\begin{equation*}
\left|\operatorname{vol}(C'_{\eps},\,\lambda^{-2}\psi^{*}g(t))-V\right|\to0\qquad\textrm{as $\eps\to0$.}
\end{equation*}
Fix $\eps>0$ small enough such that $C'_\eps$ exists and such that
\begin{equation} \label{eq_vol_Ceps}
 \vol ( C'_\eps, \lambda^{-2} \psi^* g(t) ) \leq 2V.
\end{equation}

Next, using \eqref{hermosa}, the adjunction formula,
and the fact that $\psi(C'_{\eps})$ is $J$-holomorphic with self-intersection $\alpha$ in $N$, we compute that
\begin{equation*}
\begin{split}
\operatorname{vol}(C'_{\eps},\,\lambda^{-2}\psi^{*}g(t))&=\lambda^{-2}\operatorname{vol}(C'_{\eps},\,\psi^{*}g(t))=\lambda^{-2}\operatorname{vol}(\psi(C'_{\eps}),\,g(t))\\
&=\lambda^{-2}\int_{\psi(C'_{\eps})}\left[\left.\omega(t)\right|_{\psi(C'_{\eps})}\right]\\
&=\lambda^{-2}\int_{\psi(C'_{\eps})}\left(\left[\left.\omega(0)\right|_{\psi_{\varepsilon}(C'_{\eps})}\right]-2\pi tc_{1}\left(-K_{N}|_{\psi(C'_{\eps})}\right)\right)\\
&=\lambda^{-2}\int_{\psi(C'_{\eps})}\left(\left(\left[\left.\omega(0)\right|_{\psi_{\varepsilon}(C'_{\eps})}\right]-2\pi Tc_{1}\left(-K_{N}|_{\psi(C'_{\eps})}\right)\right)+2\pi(T-t)c_{1}\left(-K_{N}|_{\psi(C'_{\eps})}\right)\right)\\
&=\lambda^{-2}\left(2(\alpha+2)\pi(T-t)+\lim_{s\to T^{-}}\operatorname{vol}(\psi(C'_{\eps}),\,g(s))\right)\\
&\geq 2\lambda^{-2}\pi(T-t).
\end{split}
\end{equation*}
Combined with the bound \eqref{eq_vol_Ceps}, we derive \eqref{eq_lambda_bound}.
\end{proof}

\begin{Lemma}
The flow $(g(t))_{t \in [0,T)}$ is Type I, i.e., there is a constant $C < \infty$ such that
\[ |{\Rm}| \leq \frac{C}{T-t}. \]
\end{Lemma}

\begin{proof}
Suppose not, i.e., there exists a sequence $(x_i, t_i) \in N \times [0,T)$ with
\begin{equation} \label{eq_Type_II}
(T-t_i) |{\Rm}|(x_i,t_i) \to \infty.
\end{equation}
Using the language\footnote{Note that in this paper we are using the Ricci flow equation $\partial_t g(t) = - \Ric_{g(t)}$ which is more common in the K\"ahler setting.
By a simple reparameterization of the time parameter, any such flow can be converted to a Ricci flow satisfying $\partial_t g(t) = - 2\Ric_{g(t)}$, which is the subject of \cite{Bamler_RF_compactness,Bamler_HK_RF_partial_regularity}.}
 of \cite{Bamler_RF_compactness,Bamler_HK_RF_partial_regularity}, consider the sequence of parabolic Type I rescalings based at $(x_i,t_i)$:
\begin{equation} \label{eq_rescl_N}
 \big(N, ((T-t_i)^{-1} g((T-t_i) t + t_i))_{-\frac{t_i}{T-t_i} \leq t \leq 0}, (\nu_{x_i,0;t}) \big).
\end{equation}
Here $(\nu_{x_i,0;t})$ denotes the conjugate heat kernel measure based at $(x_i,0)$ of the rescaled flow; this corresponds to the conjugate kernel measure based at $(x_i,t_i)$ of the original flow.
By \cite{Bamler_RF_compactness,Bamler_HK_RF_partial_regularity} and after passing to a subsequence, these metric flow pairs $\IF$-converge to an ancient metric flow pair
\[ (\XX, (\nu_{x_\infty;t})_{t \leq 0} ) \]
which is smooth away from a codimension 4 singular set.
Due to \eqref{eq_Type_II}, this flow is singular at $x_\infty$, meaning that the pointed Nash-entropy satisfies $\lim_{\tau \to 0} \mathcal{N}_{x_\infty}(\tau) < 0$ (see also \cite[Theorem~10.2]{Bamler_HK_entropy_estimates}).
So the metric flow $\XX$ has a non-trivial (i.e., non-Euclidean) tangent flow at $x_\infty$ \cite[Theorem~2.40]{Bamler_HK_RF_partial_regularity}:
\[ (\XX', (\nu_{x'_\infty;t})_{t \leq 0} ), \]
i.e., there is a sequence $\lambda_j \to 0$ such that
\begin{equation} \label{eq_j_convergence}
 \lambda_j^{-1}(\XX, (\nu_{x_\infty;t})_{t \leq 0} ) \xrightarrow[j\to\infty]{\IF} (\XX', (\nu_{x'_\infty;t})_{t \leq 0} ).
\end{equation}
By \cite[Theorems~2.40, 2.18, 2.46]{Bamler_HK_RF_partial_regularity}, this limit is a metric soliton and all of its time-slices are homothetic to a non-flat, smooth shrinking gradient Ricci soliton orbifold $(M', g', f')$ with isolated singularities.
The smooth convergence in \eqref{eq_j_convergence} on the regular part of $\XX'$ implies that there is an exhaustion $U_1 \subset U_2 \subset \ldots \subset M'$ of the regular part of $M'$ and a sequence of diffeomorphisms onto their images $\psi_j : U_j \to \RR_{-\lambda^2_j}$ into the regular part of $\XX$ such that
\begin{equation} \label{eq_pullback_conv}
 \lambda_j^{-2} \psi_j^* g^{\RR}_{-\lambda^2_j} \xrightarrow[j\to\infty]{\;\; C^\infty_{\loc} \;\;} g'.
\end{equation}
Since we have smooth convergence of the flows \eqref{eq_rescl_N} on the regular part $\RR$ of $\XX$, this implies that for any $j$ there is an $\underline{i}(j) < \infty$ such that for all $i \geq \underline{i}(j)$ there is a diffeomorphism onto its image $\phi_{j,i} : U_j \to N$ with
\begin{equation} \label{eq_conv_elementary}
 \lambda_j^{-2} (T-t_i)^{-1} \phi_{j,i}^* g(t_i - \lambda_j^2 (T-t_i)) \xrightarrow[j\to\infty, \; i \geq \underline{i}(j)]{C^\infty_{\loc}} g'.
\end{equation}
This means that the $C^\infty_{\loc}$-convergence from \eqref{eq_conv_elementary} holds for any sequence $i_j \geq \underline{i}(j)$.
It also implies that $g'$ is K\"ahler on the regular part of $M'$.

If $M'$ had an orbifold singularity, then a blowup near each singularity would be of the form $\IR^4/\Gamma$ for $\Gamma$ as in Lemma~\ref{Lem_no_cone}.
Combined with \eqref{eq_conv_elementary}, this would produce a sequence as in Lemma~\ref{Lem_no_cone} and therefore a contradiction.
Consequently $M'$ is smooth and \eqref{eq_conv_elementary} implies that we can recover $(M',g')$ as a smooth, pointed Cheeger-Gromov limit of the rescaled metrics $\lambda_j^{-2} (T-t_i)^{-2} g(t_i - \lambda_j^2 (T-t_i))$.
Moreover, since the scalar curvature of $g(t)$ is uniformly bounded from below, we have $R \geq 0$ on $(M',g')$.
If we had $R = 0$ at some point of $M'$, then the strong maximum principle applied to the evolution equation for $R$ would
imply that $\Ric = 0$ and $\nabla^2 f = g$, which would lead to a contradiction as in the proof of Lemma~\ref{Lem_cylinder_limit}.
As a result, $R > 0$ on $M'$.

\begin{Claim}
$M'$ is non-compact.
\end{Claim}

\begin{proof}
Suppose that $M'$ was compact.
Then we would have $U_j = M'$ for large $j$ and the convergence \eqref{eq_conv_elementary} would imply that there was a sequence of times $t'_j \to T$ such that $R_{\min} (t) := \min_N R(\cdot, t)$ satisfies
\[ \lim_{j \to \infty} (T-t'_j) R_{\min}(t'_j) = \infty, \]
because $\min_{M'} R > 0$.
However, the evolution equation
\[ \partial_t R = \triangle R + |{\Ric}|^2 \geq \triangle R + \tfrac14 R^2 \]
implies that $\partial_t R_{\min} (t) \geq \frac14 R_{\min}^2 (t)$, from which it follows that $R_{\min}(t) \leq 4/(T-t)$.
\end{proof}

\begin{Claim} \label{Cl_no_S2R2}
There are no sequences $y_k \in M'$ and $\lambda'_k > 0$ with $\limsup_{k\to\infty} \lambda'_k < \infty$ such that $(M'_\infty, \lb \lambda^{\prime -2}_k g', \lb y_k)$ converges to the round cylinder $\IC \times \mathbb{P}^1$ in the smooth, pointed Cheeger-Gromov sense.
In particular, $(M', g')$ is not isometric to $\IC \times \mathbb{P}^1$.
\end{Claim}

\begin{proof}
Suppose for sake of a contradiction that there were such sequences. Then we can use \eqref{eq_conv_elementary} for some sequences $j_k$, $i_k \geq \underline{i}(j_k)$ to argue that we have smooth, pointed  Cheeger-Gromov convergence of
\[ (N, \lambda_{j_k}^{-2} (T-t_{i_k})^{-1} \lambda^{\prime -2}_k g(t_{i_k} - \lambda_{j_k}^2 (T-t_{i_k})) , \phi_{j_k,i_k} (y_k)) \]
to $\IC \times \mathbb{P}^1$.
Here we need to choose $j_k$ sufficiently large based on the location of $y_k$ and then $i_k \geq \underline{i}(j_k)$.
We may also ensure that $j_k \to \infty$.
By Lemma~\ref{bonita}, this implies that for large $k$ we have
\[ \lambda_{j_k} (T-t_{i_k})^{1/2} \lambda'_k > \big(T - (t_{i_k} - \lambda_{j_k}^2 (T - t_{i_k}) \big)^{1/2} \geq (T-t_{i_k})^{1/2}. \]
This, however, contradicts the fact that $\lambda_{j_k} \to 0$ and $\limsup_{k\to\infty} \lambda'_k < \infty$.
\end{proof}

\begin{Claim}
$(M',g')$ has bounded curvature.
\end{Claim}

\begin{proof}
Assume for sake of a contradiction that $|{\Rm}|(y_k) \to \infty$ along some sequence $y_k \in M' = \RR'_{-1}$.
Thus, for any $\tau > 0$ we have $\limsup_{k\to\infty}\NN_{y_k}(\tau) < 0$ \cite[Theorem~10.2]{Bamler_HK_entropy_estimates} and
so by monotonicity of the pointed Nash-entropy, we can find sequences $A_k \to \infty$, $\tau_k \to 0$ such that \[ |\NN_{y_k}(A_k\tau_k) - \NN_{y_k} (A_k^{-1} \tau_k)| \to 0, \qquad \limsup_{k\to\infty} \NN_{y_k}(\tau_k) < 0. \]
Consider now the sequence of rescaled metric flow pairs
\[ \tau_k^{-1} (\XX', (\nu_{y_k;t})_{t \leq 0} ). \]
After passing to a subsequence, this sequence converges to another metric flow pair $(\XX'', (\nu_{y_\infty;t})_{t \leq 0} )$, which must be a metric soliton. Let $(M'', g'', f'')$
be the associated orbifold.
The metric on the regular part of $(M'', g'')$ arises as a smooth local limit of rescalings of pullbacks of the metrics $g'_{-1-\tau_k}$, akin to \eqref{eq_pullback_conv}.
As these metrics are homothetic to the metric $g'$ on $M'$, this implies that the metric on the regular part of $(M'', g'')$ is a pointed blowup limit of $(M', g')$ along a certain sequence of points converging to infinity.
Combining this with \eqref{eq_conv_elementary}, we may again rule out orbifold singularities as before. Hence $M''$ is smooth.
By Lemma~\ref{Lem_cylinder_limit} this implies that $(M'', g'')$ must be isometric to $\IC \times \mathbb{P}^1$, contradicting Claim~\ref{Cl_no_S2R2}.
\end{proof}

\begin{Claim} \label{Cl_no_P1_m1}
$M'$ does not contain a $J'$-holomorphic $\mathbb{P}^1$ with self-intersection $-1$, where $J'$ is a sublimit of the pullbacks $\phi_{j,i_j}^* J$ for some sequence $i_j \geq \underline{i}(j)$.
\end{Claim}

\begin{proof}
We argue as in the proof of Claim~\ref{Cl_no_S2R2}.
If the claim were false, then by Lemma~\ref{bonita} there would be a uniform constant $c > 0$ such that for large $j$ and $i \geq \underline{i}(j)$ we have
\[ \lambda_j (T-t_i)^{1/2} \geq c (T - t_i)^{1/2}. \]
This contradicts the fact that $\lambda_j \to 0$.
\end{proof}

So $(M',g',f')$ is a non-compact, complete, shrinking gradient K\"ahler-Ricci soliton with bounded curvature
that is not biholomorphic to $\mathbb{C}\times\mathbb{P}^1 $ and does not contain a $J'$-holomorphic $\mathbb{P}^1$ with self-intersection $-1$.
This contradicts \cite[Theorem A(i)]{ccd} and completes the proof of the lemma.
Alternatively, we can first argue that $|{\Rm}| \to 0$ at infinity on $M'$, because otherwise $(M',g',p_i)$ would  converge in the pointed Cheeger-Gromov sense to a cylinder for some $p_i \to \infty$,
which would contradict Claim~\ref{Cl_no_S2R2} (see for example \cite[Corollary 4.1]{naber}). Therefore by \cite[Theorem E(3)]{cds}, $(M',g',f')$ must be the shrinking gradient K\"ahler-Ricci soliton
 of Feldman-Ilmanen-Knopf \cite{FIK}, which contains a $J'$-holomorphic $\mathbb{P}^1$ with self-intersection $-1$. This contradicts Claim~\ref{Cl_no_P1_m1}.
\end{proof}

The following lemma concludes the proof of Theorem \ref{mainthm}.

\begin{Lemma}
There exists a complete shrinking gradient K\"ahler-Ricci soliton with bounded curvature on $\operatorname{Bl}_{x}(\mathbb{C} \times \mathbb{P}^1)$ whose associated flow appears as the pointed blowup limit centered at any point of $E$ of the rescaled flows $g_i(t):=\lambda^{-2}_i g(T+\lambda^{2}_i t),\,t \in [-\lambda_i^{-2} T ,0)$, for every sequence $\lambda_i \to 0$.
\end{Lemma}

\begin{proof}
Let $q\in E$ and consider the pointed K\"ahler-Ricci flows $(N,\,g_{i}(t),\,q)$. By the previous lemma, the K\"ahler-Ricci flow is Type I, therefore from \cite{Sesum_conv_to_soliton,Cao_Zhang_conj_h_e,topping, naber} we know that a subsequence of the flows
converges smoothly to a pointed flow $(M_\infty, (g_\infty(t)), q_\infty)$, which is the flow associated to the shrinking
gradient Ricci soliton $(M_\infty,g_\infty := g_\infty(-1), f_\infty)$ with strictly positive scalar curvature
for some smooth soliton potential $f_\infty$.
Moreover, we have smooth pointed Cheeger-Gromov convergence of the time-slices $(M, g_i(-1) = \lambda^{-2}_i g(T-\lambda^{2}_i),q)$ to $(M_\infty, g_\infty, q_\infty)$.
Thus, $(M_\infty, g_\infty)$ is K\"ahler with respect to some complex structure $J_\infty$.

If $M_\infty$ is compact, then by the smooth Cheeger-Gromov convergence we have that
\[ \lambda_i^{-4} \vol (N, g(T-\lambda_i^2)) =  \vol (N,\lambda_i^{-2} g(T- \lambda_i^2)) \xrightarrow[i \to \infty]{} \vol (M_\infty, g_\infty) < \infty. \]
This contradicts Lemma~\ref{volume} and so $M_\infty$ is non-compact.

Next we claim that there is an integral curve of $\nabla f_\infty$ along which the scalar curvature on $M_\infty$ does not tend to zero.
If this were not the case, then the scalar curvature would tend to zero globally \cite[Lemma 2.7]{ccd}, hence being non-flat,
$(M_\infty, g_\infty, f_\infty)$ would be the $U(2)$-invariant shrinking gradient K\"ahler-Ricci soliton of Feldman-Ilmanen-Knopf \cite{FIK}; see \cite[Theorem E(3)]{cds}.
This in turn would imply that $\operatorname{vol}_{t\to T^{-}}(N,\,g(t))>0$ \cite[Theorem B]{ccd} in contradiction to Lemma~\ref{volume}.

\cite[Theorem A]{ccd} and \cite[Corollary C]{charlie} now apply and tell us that $(M_\infty, g_\infty)$ is either the cylinder $\mathbb{C}\times\mathbb{P}^{1}$ or
a shrinking gradient K\"ahler-Ricci soliton on $\operatorname{Bl}_{x}(\mathbb{C} \times \mathbb{P}^1)$ with bounded curvature.
Thus, to prove the lemma, it suffices to rule out the case $\mathbb{C}\times\mathbb{P}^{1}$.
So for sake of a contradiction, assume that $(M_\infty, g_\infty)$ is the cylinder $\mathbb{C}\times\mathbb{P}^{1}$.
Then by the smooth Cheeger-Gromov convergence, there is an exhaustion $U_1 \subset U_2 \subset \ldots \subset M_\infty$ and a sequence of diffeomorphisms onto their image $\psi_i : U_i \to N$ such that $\psi_i^* g_i(-1)$ converges to $g_\infty$ locally smoothly and such that $\psi_i (q_\infty) = q$.
Let $J_i := \psi_i^* J$, where $J$ denotes the complex structure on $N$.
Then, as we have seen in the proof of Lemma \ref{bonita}, we may assume without loss of generality that $J_{i}\to J_\infty$ locally smoothly as $i\to\infty$, where $J_\infty$ denotes the complex structure on $\IC \times \mathbb{P}^1$.

Let $\widehat C_\infty = \{ z \} \times \mathbb{P}^1 \subset M_\infty$, $z\in\mathbb{C}$, denote the unique $J_\infty$-holomorphic curve passing through $q_\infty$ with trivial self-intersection. Then by \cite[Corollary 2.3]{ccd}, there exists a sequence of $J_{i}$-holomorphic curves $\widehat C_{i}$ with $q_\infty \in \widehat C_{i}$ and $\widehat C_{i}. \widehat C_{i}=0$ converging smoothly to $\widehat C_\infty$ as $i\to\infty$. (Here we have applied
\cite[Proposition 3.3.5 and Section B.4]{dusa2} to deduce the smooth convergence.)  This yields a sequence of $J$-holomorphic curves $C_{i}:=\psi_{i}(\widehat C_{i})$ in $N$ with $C_i . C_i = 0$. Using this together with the evolution $[\omega(t)]$ of the K\"ahler class as dictated by \eqref{hermosa}, we compute that
\begin{equation}\label{contradiction}
\begin{split}
\operatorname{vol}(\widehat C_{i},\psi_{i}^{*}g_{i}(-1))&=\operatorname{vol}(C_{i},\lambda^{-2}_{i}g(T-\lambda^{2}_{i} ))
=\lambda^{-2}_{i}\operatorname{vol}(C_{i},g(T-\lambda^{2}_{i} ))\\
&=\lambda^{-2}_{i}\int_{C_{i}}\left[\left.\omega(T-\lambda^{2}_{i})\right|_{C_{i}}\right]\\
&=\lambda^{-2}_{i}\int_{C_{i}}\left(\left[\left.\omega(0)\right|_{C_{i}}\right]-2\pi(T-\lambda^{2}_{i})c_{1}\left(-K_{N}|_{C_{i}}\right)\right)\\
&=\lambda^{-2}_{i}\int_{C_{i}}\left(\left[\left.\omega(0)\right|_{C_{i}}\right]-2\pi Tc_{1}\left(-K_{N}|_{C_{i}}\right)\right)+2\pi \int_{C_{i}}c_{1}\left(-K_{N}|_{C_i}\right)\\
&=\lambda^{-2}_{i}\lim_{s\to T^{-}}\operatorname{vol}(C_{i},{g(s)})+4\pi .\\
\end{split}
\end{equation}
Since the left-hand side of \eqref{contradiction} converges to $\vol(\widehat C_\infty, g_\infty) = 2\pi$ as $i \to \infty$, this implies that
\begin{equation} \label{eq_limlim}
 \lim_{i \to \infty} \lim_{s\to T^{-}}\operatorname{vol}(C_{i},{g(s)}) = 0.
\end{equation}

Write $[C_{i}]=\alpha_{1,i}[D_{1}]+\alpha_{2,i}[D_{2}]+\beta_{i}[E]$ for $\alpha_{1,i},\alpha_{2,i},\beta_{i} \in \IZ$.
Using \eqref{kahlerclass}, we compute that for $s\in[0,\,T),$
\begin{equation*}
\operatorname{vol}(C_{i},g(s))=\alpha_{1,i}(a_{2}-2s)+\alpha_{2,i}(a_{1}-2s)+\beta_{i}(b-s)
\xrightarrow[s\to T^{-}]\,\alpha_{2,\,1}(a_{1}-a_{2}),
\end{equation*}
where we recall that $T=\frac{a_{2}}{2} = b$ and $a_{1}>a_{2}$.
Combining this with \eqref{eq_limlim}, the fact that $\alpha_{2,i} \in \IZ$ implies that $\alpha_{2,i} = 0$ for large $i$.

Now, the fact that $C_{i}.C_{i}=0$ implies that for large $i$,
\[ \beta_i^2 = 2\alpha_{1,i}\alpha_{2,i} = 0. \]
However, since $C_i$ intersects $E$ in $q$, we know that
\[ 0 < E . C_i  = \beta_i. \]
This yields the desired contradiction.
\end{proof}

\bibliographystyle{amsalpha}

\bibliography{bibliography}

\end{document}